\documentclass[a4paper]{amsart}
\usepackage{color,graphicx}
\usepackage{amssymb}
\usepackage{amsmath}
\usepackage{latexsym}
\usepackage{amsthm}
\usepackage{autograph,eepic,epic,latexsym,bezier,amsbsy,color,
enumerate,amsfonts,amsmath,amscd,amssymb}
\usepackage[latin1]{inputenc}

\setloopdiam{7}\setprofcurve{7}

\newtheorem{theorem}{Theorem}

\newtheorem{proposition}[theorem]{Proposition}

\theoremstyle{definition}

\newtheorem{definition}[theorem]{Definition}
\newtheorem{example}[theorem]{Example}

\newcommand{\ad}{{\operatorname{ad}\,}}
\newcommand{\K}{\mathbb{K}}

\theoremstyle{remark}
\newtheorem{remark}[theorem]{Remark}

\begin{document}

\title{Posetted trees and Baker-Campbell-Hausdorff product}

\author{Donatella Iacono}
\address{\newline  Universit\`a degli Studi di Bari,
\newline Dipartimento di Matematica,
\hfill\newline Via E. Orabona 4,
I-70125 Bari, Italy.}
\email{iacono@dm.uniba.it}
\urladdr{www.dm.uniba.it/~iacono/}

\author{Marco Manetti}
\address{\newline 
Universit\`a degli studi di Roma ``La Sapienza'',\hfill\newline
Dipartimento di Matematica \lq\lq Guido
Castelnuovo\rq\rq,\hfill\newline
P.le Aldo Moro 5,
I-00185 Roma, Italy.}
\email{manetti@mat.uniroma1.it}
\urladdr{www.mat.uniroma1.it/people/manetti/}

\begin{abstract}
We introduce the combinatorial notion of posetted trees 
and we use it in order to write an explicit expression of 
the Baker-Campbell-Hausdorff formula.
\end{abstract}



\maketitle

\section{Introduction}

If $a,b$ are continuous operators on a Hilbert space, we may write
\[ e^ae^b=e^{a\bullet b},\qquad a\bullet b=a+b+\sum_{n=2}^{\infty}w_n(a,b),\]
where $w_n$ is a universal, non commutative, homogeneous polynomial  of degree $n$ with rational coefficients.
The product $\bullet$ is called, after \cite{Baker,Campbell,Hausdorff}, Baker-Campbell-Hausdorff (BCH) product: it is  associative and the BCH theorem asserts that every polynomial $w_n$ is a Lie element, i.e., is a linear combination of nested commutators. However, the proof of the BCH theorem does not give directly an explicit description of $w_n$ as a Lie element; 
moreover, such description is not unique in view of the Jacobi identity.

The most famous explicit expression of $a \bullet b$, in terms of nested commutators, is probably the one due to E. Dynkin
(see \cite[Equation 1]{Dynkin} or \cite[Equation 1.7.3]{DK}):
\begin{equation}\label{equ.dynkinformula}
a\bullet b=\sum_{n>0}
\frac{(-1)^{n-1}}{n}  \sum \frac{1}{p_{1}!q_{1}!\ldots p_{n}!q_{n}!}ad(a)^{p_{1}}ad(b)^{q_{1}}\ldots ad(a)^{p_{n}}ad(b)^{q_{n}-1}b,
\end{equation}
where $ad(x)=[x,-]$ is the adjoint operator, the second sum is over all possible combinations of $p_1, q_1, \ldots , p_k, q_k \in \mathbb{N}$ such that $p_i +q_i  > 0$, for $i = 1, \ldots ,k$, and $\sum_{i=1}^k (p_i +q_i)= n$.\par

The literature about Baker-Campbell-Hausdorff formula is huge. For instance: in 1998,  V. Kathotia \cite{kathotia} derived a trees summation expression for the BCH product over the real numbers
using M. Kontsevich's universal formula for deformation quantization of Poisson manifolds; 
the coefficientf of this formula are certain integrals on configuration spaces and it is 
still unknown if they are rational numbers.  
In the papers \cite{FioMan} and \cite{FiMaMa}, the authors recognize  the equation $a\bullet b\bullet c=0$ as the
Maurer-Cartan equation of the canonical $L_{\infty}$ structure on  the conormalized 
complex of  singular cochains, on the standard two dimensional simplex with values in a 
Lie algebra. Therefore, the possibility of an explicit description of  $a\bullet b$,
again as a trees  summation formula, by using the standard tools of homological perturbation 
and homotopy transfer theory \cite{LodayVallette}. The reader may also consult \cite{WS} for a 
list of explicit and recursive formulas.

The aim of this paper is to give a simple and elementary combinatorial description 
of the polynomial $w_n$ that uses some notions about planar rooted trees. The necessary 
combinatorial background is summarized in Sections \ref{sec.subroot} and \ref{sec.posetted}.
 In particular, for every finite planar rooted tree $\Gamma$, the set of its leaves admits
  a total ordering  (from left to right) and also a partial ordering $\preceq$, which takes
 care of the position of leaves with respect to the subroots.  Then, we define a posetted
 tree as a finite planar rooted tree, whose leaves are labelled by elements on a partially 
ordered set (a poset), monotonically with respect to $\preceq$.

Our main result (Theorem~\ref{thm.bchtrees})
gives an explicit description of every $w_n$ as a linear combination with rational 
coefficients of nested commutators, indexed by a certain set of posetted trees with $n$ 
leaves. The formula of the coefficients involves the Bernoulli numbers and is completely
 described in terms of the combinatorial data of   posetted trees.

\section{Subroots of planar rooted trees}
\label{sec.subroot}

This section is devoted to introduce the notion, already known in the parallel logic 
programming community
\cite{informatica}, of subroots of a planar rooted trees.

Recall that a tree is called a \emph{rooted tree} if one vertex has been designated 
the \emph{root}. Every rooted tree has a natural structure of directed tree such that,
for every vertex $u$, there exists a unique directed path from $u$ to the root.
 We shall write $u\to v$ if the vertex $v$ belongs to  the directed path from $u$ to the root.
A \emph{leaf} is a vertex without incoming edges: equivalently, a vertex $u$ is a leaf if
 the relation $v\to u$ implies $u=v$. A vertex is called \emph{internal} if it is not a leaf;
 notice that, if a rooted  tree has at least two vertices, then the root is an internal 
vertex.

\begin{center}
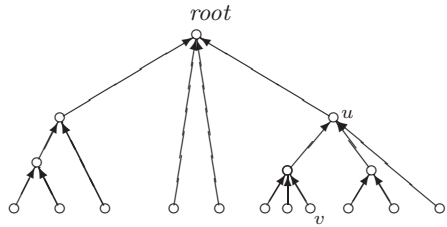
\begin{figure}[h]
\begin{picture}(200,70)
\unitlength=0.20mm

\letvertex A=(200,125)\letvertex B=(110,70)
\letvertex C=(95,40)\letvertex D=(80,10)
\letvertex E=(110,10) \letvertex F=(140,10)

\letvertex J=(185,10) \letvertex K=(215,10)

\letvertex H=(245,10)\letvertex L=(275,10)\letvertex M=(300,10)
\letvertex N=(330,10)
\letvertex O=(315,35) \letvertex P=(260,35) \letvertex Q=(290,70)
\letvertex Z=(260,10) \letvertex T=(360,10)

\drawedge(T,Q){}
\drawedge(Z,P){} \drawedge(J,A){} \drawedge(K,A){}
\drawedge(B,A){} \drawedge(E,C){}
\drawedge(F,B){}
\drawedge(Q,A){} \drawedge(C,B){}
\drawedge(D,C){} \drawedge(E,C){}

\drawedge(F,B){} \drawedge(P,Q){}
\drawedge(H,P){}\drawedge(L,P){}
\drawedge(M,O){}\drawedge(O,Q){} \drawedge(N,O){}
\drawvertex(A){$\circ$}\drawvertex(B){$\circ$}
\drawvertex(C){$\circ$} \drawvertex(P){$\circ$}

\drawvertex(D){$\circ$}

\drawvertex(E){$\circ$}\drawvertex(F){$\circ$}
\drawvertex(N){$\circ$}
\drawvertex(L){$\circ$}
 \drawvertex(M){$\circ$}
\drawvertex(H){$\circ$}\drawvertex(O){$\circ$}

\drawvertex(P){$\circ$}
\drawvertex(Q){$\circ$}
\drawvertex(Z){$\circ$} \drawvertex(J){$\circ$} \drawvertex(K){$\circ$}
\drawvertex(T){$\circ$}
\put(277,0){$\scriptstyle{v}$}
\put(295,70){$\scriptstyle{u}$}
\put(195,135){$\small\textit{root}$}

\end{picture}
\caption{A rooted tree, with  $v \to u$.
}
\end{figure}
\end{center}

From now on, we consider only planar rooted trees;  following \cite{operads}, we denote by
$\mathcal{T}$ the set of finite planar rooted trees with the root at the top and  
the  leaves at the bottom (i.e., every directed path moves upward),  and such that 
every internal  vertex has at least two incoming edges.

We also write
\[
\mathcal{T}=\bigcup_{n>0} \mathcal{T}_n,\;\]
where   $\mathcal{T}_n$ is the set of planar  rooted trees with $n$ leaves
and, for every $\Gamma\in \mathcal{T}$, we denote by
$L(\Gamma)$ the set of leaves of $\Gamma$.
The planarity of the tree gives, for every internal vertex $v$, a total ordering of the edges
ending in $v$, from the leftmost to the rightmost (see Figure~\ref{fig.orientazione}).

\begin{center}
\begin{figure}[h]
\begin{picture}(250,75)
\unitlength=0.20mm

\letvertex A=(240,125)\letvertex B=(110,70)
\letvertex C=(95,40)\letvertex D=(80,10)
\letvertex E=(110,10) \letvertex F=(140,10)
\letvertex G=(176,10)
\letvertex H=(200,10)\letvertex L=(230,10)\letvertex M=(245,10)
\letvertex N=(275,10) \letvertex O=(260,35) \letvertex P=(215,35) 
\letvertex Q=(237,71) \letvertex R=(240,95)

\letvertex a=(390,70) \letvertex b=(360,10) \letvertex c=(420,10)

\letvertex Z=(320,10) \letvertex J=(95,10)

\put(222,127){$\scriptstyle{v}$}
\put(78,-5){$\scriptstyle{1}$}
\put(93,-5){$\scriptstyle{2}$}
\put(108,-5){$\scriptstyle{3}$}
\put(138,-5){$\scriptstyle{4}$}
\put(174,-5){$\scriptstyle{5}$}
\put(198,-5){$\scriptstyle{6}$}
\put(228,-5){$\scriptstyle{7}$}
\put(243,-5){$\scriptstyle{8}$}
\put(273,-5){$\scriptstyle{9}$}
\put(318,-5){$\scriptstyle{10}$}
\put(358,-5){$\scriptstyle{11}$}
\put(418,-5){$\scriptstyle{12}$}

\drawundirectededge(a,A){} \drawundirectededge(a,b){}

\drawundirectededge(a,c){}

\drawundirectededge(Z,R){} \drawundirectededge(C,J){}

\drawundirectededge(A,B){} \drawundirectededge(E,C){}
\drawundirectededge(B,F){} \drawundirectededge(R,G){}
 \drawundirectededge(E,C){}
\drawundirectededge(B,F){} \drawundirectededge(R,G){}
\drawundirectededge(P,L){}
\drawundirectededge(O,M){}\drawundirectededge(A,R){}
\drawundirectededge(Q,R){}
\drawundirectededge(Q,O){}
\drawundirectededge(O,N){}
\drawundirectededge(P,H){} \drawundirectededge(Q,P){}

\drawundirectededge(C,D){}
\drawundirectededge(B,C){}

\drawvertex(A){$\circ$}\drawvertex(B){$\circ$}
\drawvertex(C){$\circ$} \drawvertex(P){$\circ$}

\drawvertex(D){$\circ$}

\drawvertex(a){$\circ$}
\drawvertex(b){$\circ$}
\drawvertex(c){$\circ$}

\drawvertex(E){$\circ$}\drawvertex(F){$\circ$}
\drawvertex(N){$\circ$}
\drawvertex(L){$\circ$}

\drawvertex(G){$\circ$}\drawvertex(M){$\circ$}
\drawvertex(H){$\circ$}\drawvertex(O){$\circ$}

\drawvertex(P){$\circ$}\drawvertex(R){$\circ$}
\drawvertex(Q){$\circ$}

\drawvertex(Z){$\circ$}
\drawvertex(J){$\circ$}
\end{picture}
\caption{An element of $\mathcal{T}_{12}$.}\label{fig.orientazione}
\end{figure}
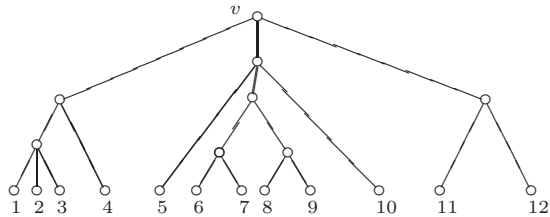
\end{center}
\begin{definition}
A \emph{rightmost branch} of a planar rooted tree  $\Gamma\in \mathcal{T}$ is a maximal connected
subgraph $\Omega\subset \Gamma$, with the property that every edge of $\Omega$ is a rightmost edge of
$\Gamma$.
A rightmost branch is called non trivial if it has at least two vertices.
\end{definition}

\begin{center}
\begin{figure}[h]

\begin{picture}(200,70)
\unitlength=0.20mm

\letvertex A=(200,125)
\letvertex B=(110,70)
\letvertex C=(95,40)
\letvertex D=(80,10) \letvertex E=(110,10) \letvertex F=(140,10)

\letvertex J=(185,10) \letvertex K=(215,10)

\letvertex H=(245,10)\letvertex L=(275,10)\letvertex M=(300,10)
\letvertex N=(330,10)
\letvertex O=(315,35) \letvertex P=(260,35) \letvertex Q=(290,70)
\letvertex Z=(260,10) \letvertex T=(360,10)

\drawundirectededge(Z,P){}

\drawundirectededge(M,O){}
\dashline[0]{2}(274,13)(262,34)
\dashline[0]{2}(329,13)(316,34)

\drawundirectededge(J,A){}

\dashline[0]{2}(358,12)(291,69)
\dashline[0]{2}(202,123)(287,72)

\drawundirectededge(K,A){}

\drawundirectededge(B,A){}


\dashline[0]{2}(109,13)(96,39)
\dashline[0]{2}(139,13)(111,69)

\drawundirectededge(C,B){}
\drawundirectededge(D,C){}

\drawundirectededge(P,Q){}
\drawundirectededge(H,P){}
\drawundirectededge(O,Q){}
\drawvertex(A){$\circ$}\drawvertex(B){$\circ$}
\drawvertex(C){$\circ$} \drawvertex(P){$\circ$}

\drawvertex(D){$\circ$}

\drawvertex(E){$\circ$}\drawvertex(F){$\circ$}
\drawvertex(N){$\circ$}
\drawvertex(L){$\circ$}
 \drawvertex(M){$\circ$}
\drawvertex(H){$\circ$}\drawvertex(O){$\circ$}

\drawvertex(P){$\circ$}
\drawvertex(Q){$\circ$}
\drawvertex(Z){$\circ$} \drawvertex(J){$\circ$} \drawvertex(K){$\circ$}
\drawvertex(T){$\circ$}
\end{picture}
\caption{An element of $\mathcal{T}_{11}$. The dashed lines denote the rightmost edges.}\label{fig.ramidestritratteggiati}
\end{figure}
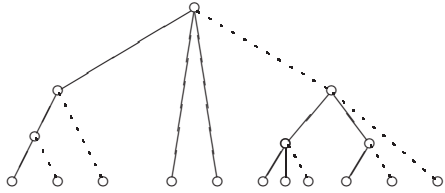
\end{center}

\begin{definition}\label{def.rigthmost m(v) d(v)}
A \emph{local rightmost leaf} is a  leaf lying on a non trivial rightmost branch.  
Given an internal vertex  $v$, we call $m(v)$ the leaf lying on the rightmost
branch containing $v$.
We also denote by $d(v)$ the distance between $v$ and $m(v)$, as defined in \cite{ore}.
\end{definition}
\begin{definition}
A \emph{subroot}  is the vertex  of a  non trivial rightmost branch which is nearest to the root. The set of subroots of a finite planar rooted tree $\Gamma$ will be denoted by $R(\Gamma)$.

\end{definition}

Therefore, we have the natural  bijections
\[ \{\text{ subroots }\}\cong\{\text{ non trivial rightmost branches }\}
\cong\{\text{ local rightmost leaves }\}.\]

\begin{example}
In the  tree of Figure~\ref{fig.esempiomassimilocali},
 the subroots are the vertices $r,a,c$ and $e$; the rightmost leaves are the leaves $2,3,5$ and $7$. Moreover, $m(a)=3,  m(c)=2, m(e)=5$ and $ m(r)=m(b)=m(f)=7$; and $d(r)= 3$,  $d(b)=2$ and $d(a)= d(c)= d(e)= d(f)= 1$.
\end{example}

\begin{center}
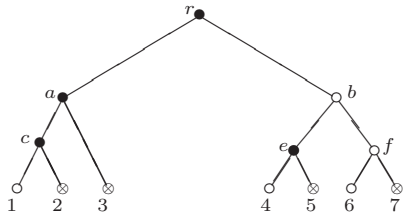
\begin{figure}[h]
\begin{picture}(200,70)
\unitlength=0.20mm

\letvertex A=(200,125)\letvertex B=(110,70)
\letvertex C=(95,40)\letvertex D=(80,10)
\letvertex E=(110,10) \letvertex F=(140,10)

\letvertex H=(246,10)\letvertex L=(275,10)\letvertex M=(300,10)
\letvertex N=(330,10)
\letvertex O=(315,35) \letvertex P=(262,35) \letvertex Q=(290,70)

\put(190,126){$\scriptstyle{r}$}
\put(98,70){$\scriptstyle{a}$}  \put(297,70){$\scriptstyle{b}$}
\put(82,40){$\scriptstyle{c}$} \put(320,35){$\scriptstyle{f}$}  \put(252,35){$\scriptstyle{e}$}

\put(73,-5){$\scriptstyle{1}$} \put(103,-5){$\scriptstyle{2}$}  \put(133,-5){$\scriptstyle{3}$} \put(240,-5){$\scriptstyle{4}$}
 \put(270,-5){$\scriptstyle{5}$} \put(295,-5){$\scriptstyle{6}$} \put(325,-5){$\scriptstyle{7}$}

\drawundirectededge(B,A){} \drawundirectededge(E,C){}
\drawundirectededge(F,B){}
\drawundirectededge(Q,A){} \drawundirectededge(C,B){}
\drawundirectededge(D,C){} \drawundirectededge(E,C){}

\drawundirectededge(F,B){} \drawundirectededge(P,Q){}
\drawundirectededge(H,P){}
\drawundirectededge(L,P){}
\drawundirectededge(M,O){}\drawundirectededge(O,Q){}
\drawundirectededge(N,O){}
\drawvertex(A){$\bullet$}\drawvertex(B){$\bullet$}
\drawvertex(C){$\bullet$} \drawvertex(P){$\bullet$}

\drawvertex(D){$\circ$}

\drawvertex(E){$\scriptscriptstyle{\otimes}$}\drawvertex(F){$\scriptscriptstyle{\otimes}$}
\drawvertex(N){$\scriptscriptstyle{\otimes}$}
\drawvertex(L){$\scriptscriptstyle{\otimes}$}
 \drawvertex(M){$\circ$}
\drawvertex(H){$\circ$}\drawvertex(O){$\circ$}

\drawvertex(P){$\circ$}
\drawvertex(Q){$\circ$}
\end{picture}
\caption{The subroots are denoted by $\bullet$, while the local rightmost leaves by 
 $ \stackrel{\scriptscriptstyle{\otimes}}{ }$.}\label{fig.esempiomassimilocali}
\end{figure}
\end{center}

A  planar rooted tree $\Gamma\in\mathcal{T}$ is a \emph{binary tree} if every internal vertex has exactly two incoming edges. We use the notation
\[
\mathcal{B}=\bigcup_{n>0} \mathcal{B}_n \subset \mathcal{T},\;
\]
where
$\mathcal{B}_n$ is the set of planar binary rooted trees with $n$ leaves.
Using the  notion introduced above, it is very easy to see that a tree $\Gamma  \in \mathcal{T}_n$ 
is a binary tree if and only if it satisfies the equality:
\[\sum_{v \in R(\Gamma)}d(v)=n-1.\]

Let  $R$  be a (non associative) algebra  over a field $\K$ and 
 $\Gamma\in \mathcal{B}$ a planar rooted tree. Labelling the leaves of $\Gamma$ 
with elements of $R$, we can associate  the product element in $R$  obtained  
by the usual operadic rules \cite{LodayVallette,operads}, i.e.,  we perform the product of $R$ 
at every internal vertex in the order arising from the planar structure of the directed tree.  
For instance, the following labelled tree

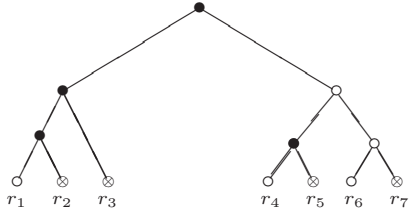
\begin{center}
\begin{picture}(200,65)
\unitlength=0.20mm

\letvertex A=(200,125)\letvertex B=(110,70)
\letvertex C=(95,40)\letvertex D=(80,10)
\letvertex E=(110,10) \letvertex F=(140,10)

\letvertex H=(245,10)\letvertex L=(275,10)\letvertex M=(300,10)
\letvertex N=(330,10)
\letvertex O=(315,35) \letvertex P=(262,35) \letvertex Q=(290,70)

\put(73,-5){$\scriptstyle{r_1}$} \put(103,-5){$\scriptstyle{r_2}$}  \put(133,-5){$\scriptstyle{r_3}$} \put(240,-5){$\scriptstyle{r_4}$}
 \put(270,-5){$\scriptstyle{r_5}$} \put(295,-5){$\scriptstyle{r_6}$} \put(325,-5){$\scriptstyle{r_7}$}

\drawundirectededge(B,A){} \drawundirectededge(E,C){}
\drawundirectededge(F,B){}
\drawundirectededge(Q,A){} \drawundirectededge(C,B){}
\drawundirectededge(D,C){} \drawundirectededge(E,C){}

\drawundirectededge(F,B){} \drawundirectededge(P,Q){}
\drawundirectededge(H,P){}\drawundirectededge(L,P){}
\drawundirectededge(M,O){}\drawundirectededge(O,Q){}
\drawundirectededge(N,O){}
\drawvertex(A){$\bullet$}\drawvertex(B){$\bullet$}
\drawvertex(C){$\bullet$} \drawvertex(P){$\bullet$}

\drawvertex(D){$\circ$}

\drawvertex(E){$\scriptscriptstyle{\otimes}$}\drawvertex(F){$\scriptscriptstyle{\otimes}$}
\drawvertex(N){$\scriptscriptstyle{\otimes}$}
\drawvertex(L){$\scriptscriptstyle{\otimes}$}
 \drawvertex(M){$\circ$}
\drawvertex(H){$\circ$}\drawvertex(O){$\circ$}

\drawvertex(P){$\circ$}
\drawvertex(Q){$\circ$}
\end{picture}
\end{center}

gives the  product $((r_1 r_2)r_3) ((r_4 r_5 )(r_6 r_7)) \in R$.

Given any  map $f: L(\Gamma) \to R$ (the labelling), we denote by $Z_{\Gamma}(f) \in R$ the corresponding product element.

If $S\subset R$, then the elements $Z_{\Gamma}(f)$, with 
$\Gamma\in\mathcal{B}$ and $f\colon L(\Gamma)\to S$,  are a 
set of generators of the subalgebra generated by $S$. If $R$ is either 
symmetric or skewsymmetric (e.g., a Lie algebra), then we may reduce the
 set of generators by a suitable choice of the labelling. Keeping in mind 
our main application (the BCH product), a possible way of doing 
that is by introducing the combinatorial notion of posetted trees.

\section{Posetted trees}
\label{sec.posetted}

Using the notion of subroot, we can define a partial order $ \preceq $ on the set of leaves $L(\Gamma)$.

\begin{definition}
Given two leaves $l_1$ and $l_2$ in a tree $\Gamma\in\mathcal{T}$, we say $l_1 \preceq l_2$ if $l_1=l_2$ or there exists a subroot $v \in R(\Gamma)$ such that $l_2=m(v)$ and  $l_1 \to v$.
\end{definition}

\begin{center}
\begin{figure}[h]
\begin{picture}
(200,60)
\unitlength=0.20mm

\letvertex B=(200,110)
\letvertex F=(115,10) \letvertex T=(285,10)

\letvertex A=(200,65)

\letvertex C=(170,30)  \letvertex E=(230,30)
\letvertex D=(155,10)  \letvertex G=(245,10)
\letvertex N=(185,10)  \letvertex L=(215,10)

\drawundirectededge(A,B){}
\drawundirectededge(T,B){}
\drawundirectededge(C,D){} \drawundirectededge(E,G){}
   \drawundirectededge(L,E){}
\drawundirectededge(B,F){} \drawundirectededge(C,N){}
\drawundirectededge(A,C){} \drawundirectededge(A,E){}

\drawvertex(T){$\scriptscriptstyle{\otimes}$}
\drawvertex(A){$\bullet$}\drawvertex(B){$\bullet$}
\drawvertex(C){$\bullet$}
 \drawvertex(D){$\circ$}
\drawvertex(E){$\circ$}\drawvertex(F){$\circ$}
\drawvertex(G){$\scriptscriptstyle{\otimes}$}
\drawvertex(N){$\scriptscriptstyle{\otimes}$}\drawvertex(L){$\circ$}

\put(103,-5){$\scriptstyle{l_1}$} \put(153,-5){$\scriptstyle{l_2}$}  
\put(183,-5){$\scriptstyle{l_3}$} \put(213,-5){$\scriptstyle{l_4}$}
 \put(243,-5){$\scriptstyle{l_5}$} \put(288,-5){$\scriptstyle{l_6}$}

\end{picture}
\caption{Here, we have  $l_1\preceq l_6$,   $l_2 \preceq l_3 \preceq l_5 \preceq l_6$ and  $l_4 \preceq l_5 $.} \label{fig.ordineFoglie}
\end{figure}
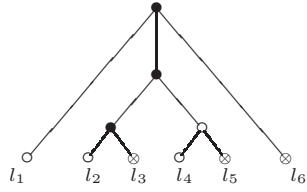
\end{center}

\begin{definition} For every
poset $(A ,\leq)$, we denote
\[
 \mathcal{T}(A)=\{ (\Gamma, f) \, | \, \Gamma \in \mathcal{T}, \, f:(L(\Gamma), \preceq) \to (A ,\leq),
 f \mbox{ monotone}\}
\]
In a similar way, we define $\mathcal{B}(A)$, and, for every $n>0$, $\mathcal{T}_n(A)$ and $\mathcal{B}_n(A)$.
\end{definition}

We  call \emph{posetted trees} the elements of $\mathcal{T}(A)$.

\begin{example}\label{example. B(a < b)}
The sets $\mathcal{B}_1(b\le a)$, $\mathcal{B}_2(b\le a)$ and $\mathcal{B}_3(b\le a)$ 
 contain 2, 3 and 8 posetted trees, respectively  (see Figures \ref{fig.b2ab} and \ref{fig.b3ab}).
\end{example}

\begin{center}

\begin{figure}[h]

\begin{picture}(240,25)
\unitlength=0.30mm
\letvertex A=(120,35)\letvertex B=(105,10)
\letvertex F=(135,10)
\letvertex p=(200,35)\letvertex q=(185,10)
\letvertex r=(215,10)
\letvertex a=(280,35)\letvertex b=(265,10)
\letvertex f=(295,10)

\drawundirectededge(A,B){}\drawundirectededge(F,A){}

\put(10,0){$\scriptstyle{a}$}
\put(50,0){$\scriptstyle{b}$}
\put(10,10){$\circ$}
\put(50,10){$\circ$}

\put(103,0){$\scriptstyle{a}$}
\put(135,0){$\scriptstyle{a}$}
\put(183,0){$\scriptstyle{b}$}
\put(215,0){$\scriptstyle{b}$}
\put(263,0){$\scriptstyle{b}$}
\put(295,0){$\scriptstyle{a}$}


\drawundirectededge(f,a){}\drawundirectededge(b,a){}
\drawundirectededge(r,p){}\drawundirectededge(q,p){}

\drawvertex(a){$\bullet$}
\drawvertex(b){$\circ$}
\drawvertex(f){$\scriptscriptstyle{\otimes}$}
\drawvertex(p){$\bullet$}
\drawvertex(q){$\circ$}
\drawvertex(r){$\scriptscriptstyle{\otimes}$}
\drawvertex(A){$\bullet$}\drawvertex(B){$\circ$}
\drawvertex(F){$\scriptscriptstyle{\otimes}$}
\end{picture}
\caption{ The 5 posetted trees of $\mathcal{B}_i(b\le a)$, $i=1,2$.}\label{fig.b2ab}
\end{figure}
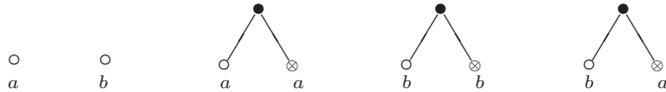
\end{center}

\begin{center}
\begin{figure}[h]
\begin{picture}(200,30)
\unitlength=0.23mm
\letvertex Aa=(-20,60)\letvertex Ba=(-35,35)
\letvertex Ca=(-50,10)\letvertex Da=(-20,10)
\letvertex Ea=(10,10)\letvertex Fa=(-5,35)

\letvertex A=(90,60)\letvertex B=(75,35)
\letvertex C=(60,10)\letvertex D=(90,10)
\letvertex E=(120,10)\letvertex F=(105,35)

\letvertex a=(200,60)\letvertex b=(185,35)
\letvertex c=(170,10)\letvertex d=(200,10)
\letvertex e=(230,10)\letvertex f=(215,35)

\letvertex AA=(310,60)\letvertex BB=(295,35)
\letvertex CC=(280,10)\letvertex DD=(310,10)
\letvertex EE=(340,10)\letvertex FF=(325,35)

\letvertex aA=(420,60)\letvertex aB=(405,35)
\letvertex aC=(390,10)\letvertex aD=(420,10)
\letvertex aE=(450,10)\letvertex aF=(435,35)

\drawundirectededge(A,F){} \drawundirectededge(F,E){}  \drawundirectededge(F,D){}
\drawundirectededge(A,C){}

\drawundirectededge(Aa,Fa){} \drawundirectededge(Fa,Ea){}  \drawundirectededge(Fa,Da){}
\drawundirectededge(Aa,Ca){}

\drawundirectededge(a,f){}\drawundirectededge(f,e){}
\drawundirectededge(a,c){}\drawundirectededge(f,d){}


\drawundirectededge(AA,CC){}\drawundirectededge(AA,EE){}
\drawundirectededge(BB,DD){}

\drawundirectededge(aA,aC){}\drawundirectededge(aA,aE){}

\drawundirectededge(aB,aD){}

\drawvertex(a){$\bullet$}
\drawvertex(c){$\circ$}\drawvertex(d){$\circ$}
\drawvertex(e){$\scriptscriptstyle{\otimes}$}\drawvertex(f){$\circ$}

\drawvertex(A){$\bullet$}
\drawvertex(C){$\circ$}\drawvertex(D){$\circ$}
\drawvertex(E){$\scriptscriptstyle{\otimes}$}\drawvertex(F){$\circ$}

\drawvertex(Aa){$\bullet$}
\drawvertex(Ca){$\circ$}\drawvertex(Da){$\circ$}
\drawvertex(Ea){$\scriptscriptstyle{\otimes}$}\drawvertex(Fa){$\circ$}

\drawvertex(AA){$\bullet$}\drawvertex(BB){$\bullet$}
\drawvertex(CC){$\circ$}\drawvertex(DD){$\scriptscriptstyle{\otimes}$}
\drawvertex(EE){$\scriptscriptstyle{\otimes}$}

\drawvertex(aA){$\bullet$}\drawvertex(aB){$\bullet$}
\drawvertex(aC){$\circ$}\drawvertex(aD){$\scriptscriptstyle{\otimes}$}
\drawvertex(aE){$\scriptscriptstyle{\otimes}$}

\put(118,0){$\scriptstyle{a}$}\put(58,0){$\scriptstyle{a}$}
\put(88,0){$\scriptstyle{a}$}

\put(-52,0){$\scriptstyle{b}$}\put(8,0){$\scriptstyle{b}$}
\put(-22,0){$\scriptstyle{b}$}

\put(388,0){$\scriptstyle{b}$}\put(448,0){$\scriptstyle{b}$}
\put(418,0){$\scriptstyle{b}$}

\put(168,0){$\scriptstyle{b}$}\put(198,0){$\scriptstyle{a}$}
\put(228,0){$\scriptstyle{a}$}
\put(308,0){$\scriptstyle{a}$}\put(338,0){$\scriptstyle{a}$}
\put(278,0){$\scriptstyle{a}$}

\end{picture}

\begin{picture}(200,50)
\unitlength=0.23mm
\letvertex A=(60,60)\letvertex B=(45,35)
\letvertex C=(30,10)\letvertex D=(60,10)
\letvertex E=(90,10)\letvertex F=(75,35)

\letvertex a=(200,60)\letvertex b=(185,35)
\letvertex c=(170,10)\letvertex d=(200,10)
\letvertex e=(230,10)\letvertex f=(215,35)

\letvertex AA=(340,60)\letvertex BB=(325,35)
\letvertex CC=(310,10)\letvertex DD=(340,10)
\letvertex EE=(370,10)\letvertex FF=(355,35)

\drawundirectededge(A,F){} \drawundirectededge(F,E){}  \drawundirectededge(F,D){}
\drawundirectededge(A,C){}

\drawundirectededge(a,f){}\drawundirectededge(f,e){}
\drawundirectededge(a,c){}\drawundirectededge(f,d){}


\drawundirectededge(AA,CC){}\drawundirectededge(AA,EE){}
\drawundirectededge(BB,DD){}

\drawvertex(a){$\bullet$}
\drawvertex(c){$\circ$}\drawvertex(d){$\circ$}
\drawvertex(e){$\scriptscriptstyle{\otimes}$}\drawvertex(f){$\circ$}

\drawvertex(A){$\bullet$}
\drawvertex(C){$\circ$}\drawvertex(D){$\circ$}
\drawvertex(E){$\scriptscriptstyle{\otimes}$}\drawvertex(F){$\circ$}

\drawvertex(AA){$\bullet$}\drawvertex(BB){$\bullet$}
\drawvertex(CC){$\circ$}\drawvertex(DD){$\scriptscriptstyle{\otimes}$}
\drawvertex(EE){$\scriptscriptstyle{\otimes}$}

\put(28,0){$\scriptstyle{a}$}\put(58,0){$\scriptstyle{b}$}
\put(88,0){$\scriptstyle{a}$}

\put(168,0){$\scriptstyle{b}$}\put(198,0){$\scriptstyle{b}$}
\put(228,0){$\scriptstyle{a}$}
\put(308,0){$\scriptstyle{b}$}\put(338,0){$\scriptstyle{a}$}
\put(368,0){$\scriptstyle{a}$}

\end{picture}
\caption{ The 8 posetted trees of $\mathcal{B}_3(b\le a)$.}\label{fig.b3ab}
\end{figure}
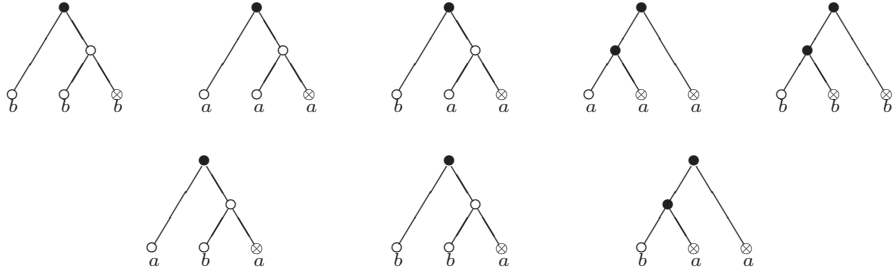
\end{center}

\begin{remark} If $A=\{1,\ldots,m\}$ with the usual order, then there exists a natural inclusion of $\mathcal{T}_n(A)$ into the set  of admissible graphs with $n$ vertices of the first kind and $m$ vertices of the second kind considered in \cite{kathotia,Konts} . 
\end{remark}

Assume that $A$ is a subset of a (skew)commutative algebra $R$ and choose a total ordering  on $A$. Then, it is easy to see that the elements $Z_{\Gamma}(f)$, with $(\Gamma,f)\in \mathcal{B}(A)$ generate, as a $\K$ vector space, the subalgebra generated by $A$.

\section{An expression of the Baker-Campbell-Hausdorff product in terms of posetted trees}

Let $L$ be a Lie algebra over a field $\K$ of characteristic $0$, which is complete
 with respect to its lower descending series $L^1=L$,  $L^{n+1}=[L^n,L]$. Denote
 by $\bullet\colon L\times L\to
L$ the Baker-Campbell-Hausdorff (BCH) product,  obtained formally by the formula
$a\bullet b=\log(e^ae^b)$. It is well known that
\[ a \bullet b  =a+b +
\displaystyle\frac{1}{2}[a,b]+\frac{1}{12} [a,[a,b]]-\frac{1}{12}
[b,[b,a]]+ \cdots,\]
is an element of the Lie subalgebra generated by $a$ and $b$ and, then, it can be expressed as an infinite sum
\[ a\bullet b=\sum_{(\Gamma,f)\in\mathcal{B}(b\le a)} s_{(\Gamma,f)}Z_{\Gamma}(f),\]
for a sequence  $s_{(\Gamma,f)}\in\K$. Clearly, in view of the alternating properties of the product and Jacobi identity,  such a sequence is not unique. The Dynkin Formula \eqref{equ.dynkinformula} provides  a  sequence as above where $s_{(\Gamma,f)}=0$, whenever  $\Gamma$ has at least 2 subroots: 
on the other hand, the explicit expression of the nonvanishing $s_{(\Gamma,f)}$ is rather complicated.

Here, we describe another sequence $b_{(\Gamma,f)}$ of rational numbers with the 
above properties. 
First of all, define the sequence of rational numbers $\{b_n\}$, for every $n\ge 0$,   
by  their  ordinary generating function
\[
\sum_{n\ge 0}b_n x^n=\frac{x}{e^x-1}\;.
\]
Notice that $b_n=B_n n!$ where the $B_n$ are the  Bernoulli numbers. 
In particular, the only non trivial odd term of the sequence is $b_1=-\frac{1}{2}$ 
and we have:
\[ b_0=1, \quad b_2=\frac{1}{12},\quad b_4=-\frac{1}{720},\quad\ldots \, .\]

\begin{definition}  Given a poset $A$ and a posetted tree $(\Gamma,f)\in \mathcal{T}(A)$,
let us define
\[
b_{(\Gamma,f)} := \prod_{v \in R(\Gamma)} \frac{b_{d(v)}}{t(v)},
\]
where the $b_n$'s are  the  rational numbers above and, for every subroot 
$v\in R(\Gamma)$, we have
\[ t(v)=\text{ number of leaves $u\in L(\Gamma)$ such that $u\to v$ and $f(u)=f(m(v))$}.\]
We remind that $m(v)$ is the leaf lying on the rightmost branch containing $v$ (Definition 
\ref{def.rigthmost m(v) d(v)}).

\end{definition}

\begin{example}
Let $A=\{ b \leq a\}$ and consider the posetted tree
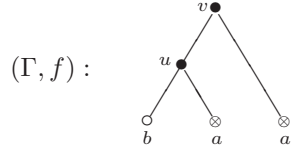
\begin{center}
\begin{picture}(300,60)
\unitlength=0.30mm
\put(110,30){$(\Gamma,f):$}
\letvertex AA=(200,60)\letvertex BB=(185,35)
\letvertex CC=(170,10)\letvertex DD=(200,10)
\letvertex EE=(230,10)\letvertex FF=(215,35)

\drawundirectededge(AA,CC){}\drawundirectededge(AA,EE){}
\drawundirectededge(BB,DD){}

\drawvertex(AA){$\bullet$}\drawvertex(BB){$\bullet$}
\drawvertex(CC){$\circ$}\drawvertex(DD){$\scriptscriptstyle{\otimes}$}
\drawvertex(EE){$\scriptscriptstyle{\otimes}$}

 \put(175,35){$\scriptstyle{u}$}\put(192,59){$\scriptstyle{v}$}
\put(168,0){$\scriptstyle{b}$}\put(198,0){$\scriptstyle{a}$} \put(228,0){$\scriptstyle{a}$}
\end{picture}
\end{center}
Here,  we have    $d(u)=d(v)=1; \  t(u)=1; \  t(v)=2$; therefore,
$
b_{(\Gamma,f)}=\dfrac{b_1}{ 1} \cdot \dfrac{b_1}{ \,2}=\dfrac{1}{8}.
$
\end{example}

\begin{theorem}\label{thm.bchtrees}
Let $L$ be a Lie algebra as above; then,
for every positive integer $k$ and every  $a_1, \ldots , a_k \in L$, we have
\begin{equation}\label{equ.bchperk}
a_k\bullet a_{k-1} \bullet \cdots \bullet a_1 =\sum_{(\Gamma,f) \, \in \, \mathcal{B}\,(a_1 \leq a_2 \leq  \cdots \leq a_k)} b_{(\Gamma,f)} Z_\Gamma(f),
\end{equation}
\begin{equation}\label{equ.bchperkbis}
a_1\bullet a_{2} \bullet \cdots \bullet a_k =\sum_{n=1}^{+\infty}(-1)^{n-1}\sum_{(\Gamma,f) \, \in \, \mathcal{B}_n\,(a_1 \leq a_2 \leq  \cdots \leq a_k)} b_{(\Gamma,f)} Z_\Gamma(f).
\end{equation}
In particluar, for $a,b\in L$, we have
\begin{equation}\label{equ.bchperdue}
a\bullet b=\sum_{(\Gamma,f)  \in \mathcal{B}(b \leq a)} b_{(\Gamma,f)} Z_\Gamma(f).
\end{equation}
\end{theorem}

\begin{proof}  Let us first prove Formula \eqref{equ.bchperdue}.
Let  $\mathcal{C}'\,(b\leq a)\subset \mathcal{B}\,(b\leq a)$ be the subset of posetted trees having every local rightmost leaf labelled with $a$ and denote by $\mathcal{C}\,(b\leq a)= \mathcal{C}'\,(b\leq a)\cup \mathcal{B}_1(b)$.

Since the bracket is skewsymmetric, we have that $Z_\Gamma(f)=0$, for every 
$(\Gamma,f)\notin\mathcal{C}\,(b\leq a)$; therefore,
\[  \sum_{(\Gamma,f)  \in \mathcal{B}(b \leq a)} b_{(\Gamma,f)} Z_\Gamma(f)=
\sum_{(\Gamma,f)  \in \mathcal{C}(b\leq a)} b_{(\Gamma,f)} Z_\Gamma(f).\]
In \cite[Theorem. 1.6.1]{DK} and \cite{hall}, the following recursive formula for the   Baker-Campbell-Hausdorff product  is proved:
\[a\bullet b=\sum_{r\ge 0}Z_r,\]
where
\[Z_0=b,\qquad
Z_{r+1}=\frac{1}{r+1}\sum_{m\ge 0}b_m \sum_{i_1+\cdots+i_m=r}(\ad Z_{i_1})
(\ad Z_{i_2})\cdots(\ad Z_{i_m})a,\quad \text{for }r\ge 0.\]
For every $r>0$, let
$\mathcal{C}_r\subset \mathcal{C}(b\leq a)$ be the subset of posetted trees with exactly $r$
leaves labelled with $a$; we prove  that, for every $r\ge 0$, we have
\begin{equation}\label{equ.formulazetaerre}
 Z_r=\sum_{(\Gamma,f)  \in \mathcal{C}_r} b_{(\Gamma,f)} Z_\Gamma(f).
 \end{equation}
This is clear for $r=0$;  for $r=1$, we have
\[ Z_{1}=\sum_{m\ge 0}b_m(\ad Z_{0})^m
a=\sum_{m\ge 0}b_m(\ad b)^m a,\]
whereas $\mathcal{C}_1=\{\Omega_m\}$, $m\ge 0$, is the set of posetted trees of 
Bernoulli type  \cite{torossian}, i.e.,
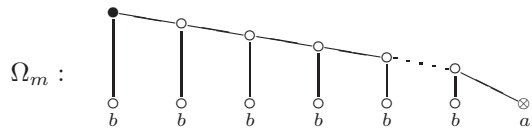
\begin{center}
\begin{picture}(290,30)
\unitlength=0.30mm
 \letvertex A=(105,45) \letvertex B=(135,40)
 \letvertex C=(165,35) \letvertex D=(195,30)
 \letvertex E=(225,25)
 \letvertex F=(255,20)
 \letvertex a=(105,5) \letvertex b=(135,5)
 \letvertex c=(165,5) \letvertex d=(195,5)
 \letvertex e=(225,5)
 \letvertex f=(255,5) \letvertex g=(285,5)

\put(60,15){$\Omega_m:$}

\drawundirectededge(A,B){}
\drawundirectededge(B,C){}
\drawundirectededge(C,D){}
\drawundirectededge(D,E){}

\drawundirectededge(F,g){}
\drawundirectededge(A,a){} \drawundirectededge(F,f){}
\drawundirectededge(B,b){}

\drawundirectededge(E,e){}
\drawundirectededge(C,c){}\drawundirectededge(D,d){}

 \dashline[0]{2}(228,25)(252,20)

\drawvertex(A){$\bullet$}\drawvertex(B){$\circ$}
\drawvertex(C){$\circ$}\drawvertex(D){$\circ$}
\drawvertex(E){$\circ$}\drawvertex(F){$\circ$}
\drawvertex(a){$\circ$}\drawvertex(c){$\circ$}
\drawvertex(b){$\circ$}\drawvertex(d){$\circ$}
\drawvertex(e){$\circ$}\drawvertex(f){$\circ$}
\drawvertex(g){$\scriptscriptstyle\otimes$}

\put(103,-5){$\scriptstyle{b}$} \put(133,-5){$\scriptstyle{b}$} 
\put(163,-5){$\scriptstyle{b}$} \put(193,-5){$\scriptstyle{b}$}
 \put(223,-5){$\scriptstyle{b}$}
\put(253,-5){$\scriptstyle{b}$} \put(283,-5){$\scriptstyle{a}$}

\end{picture}
\end{center}
where $m$ is the number of leaves labelled with $b$. Therefore, the coefficient 
 $b_{(\Omega_m)}$ is exactly 
$b_m$ and so
 \[ Z_1=\sum_{(\Gamma,f)  \in \mathcal{C}_1} b_{(\Gamma,f)} Z_\Gamma(f).\]

\noindent Moreover, every element of $\mathcal{C}_{r+1}$  is obtained in a unique way starting 
from a tree $\Omega_m$ and grafting, at each of the $m$ leaves labelled with $b$, 
the roots of elements of $\mathcal{C}_{i_1},\ldots,\mathcal{C}_{i_m}$, with $i_1+\cdots+i_m=r$ 
(for the definition of the grafting see \cite[Definition 1.37]{operads}). 
Therefore, the proof of \eqref{equ.formulazetaerre} follows  easily by induction on $r$.

Next, since $\bullet$ is associative, we have
\[ a_1\bullet a_{2} \bullet \cdots \bullet a_k=-((-a_k)\bullet\cdots\bullet(-a_1)),\]
and Formula \eqref{equ.bchperkbis} follows immediately from \eqref{equ.bchperk}.
Finally, setting $b=a_{k-1} \bullet \cdots \bullet a_1$, we have that every posetted 
tree of $\mathcal{B}\,(a_1 \leq a_2 \leq  \cdots \leq a_k)$ can be described in a unique 
way as a posetted tree in $\mathcal{C}\,(b\leq a_k)$, where at every leaf labelled 
with $b$ is grafted the root of a posetted tree of
$\mathcal{B}\,(a_1 \leq a_2 \leq  \cdots \leq a_{k-1})$. In view of the associativity relation
\[a_k\bullet a_{k-1} \bullet \cdots \bullet a_1= a_k\bullet b,
\]
we obtain that
\eqref{equ.bchperk} is a consequence of $\displaystyle a_k\bullet b=\sum_{(\Gamma,f) 
 \in \mathcal{C}(b\le a_k)} b_{(\Gamma,f)} Z_\Gamma(f)$.

\end{proof}

\begin{remark}
Choose $a_1=b$ and  $a_2=a$ in Equation \eqref{equ.bchperk}, and  $a_1=a$ and $a_2=b$ in 
Equation  \eqref{equ.bchperkbis}. Comparing the coefficient of the product
 $ad(b)^n(a)$ in both equations, we obtain the following relations 
\begin{equation}\label{equazione bernoulli number}
(1+n(-1)^n)b_n=-\sum_{i=1}^{n-1}(-1)^ib_ib_{n-i},\qquad n>0.
\end{equation} 
 Indeed, the coefficient of  $ad(b)^n(a)$ in Equation  \eqref{equ.bchperkbis} 
comes from the Bernoulli tree $\Omega_n$ and so
it  is exactly $b_n$. On the other side, we need to consider the subset $S(n)$
  of trees ${(\Gamma,f)  
\in \mathcal{B}(a \leq b)} $ with only one subroot,  $n$ leaves labelled $b$ and one leaf labelled $a$. For any $ (\Gamma,f)  \in S(n)$,
we have $Z_\Gamma(f)=\pm ad(b)^n(a)$, and we can define $C_n$ as
\[ \sum_{(\Gamma,f)\in S(n)} b_{(\Gamma,f)} Z_\Gamma(f)=C_n ad(b)^n(a).\]
Comparing the coefficients, we have 
\[C_n=(-1)^n b_n.\]
Next, let us compute  $C_n$ recursively. There are two different types of  
contributions to $C_n$ due to the 
following graphs. 
The first contribution is due to the graph with only one subroot; 
in this case, the coefficient is 
$\dfrac{b_n}{n} ad(b)^{n-1}([a,b])=-\dfrac{b_n}{n} ad(b)^{n}(a).$
The other contribution is due to the graphs obtained   from a graph in $S(i)$, 
for every  $i=1,\ldots,n-1$, and grafting, at the leaves labelled with $a$, a graph  
of $S(n-i)$.
Therefore, for every fixed $i$, the coefficient is  
\[ \frac{b_i}{n}C_{n-i} ad(b)^{i-1}([ad(b)^{n-i}(a),b])=-\frac{b_i}{n}C_{n-i} ad(b)^{n}(a).\]
Summing up, we have
\[ C_n=-\dfrac{b_n}{n}-\sum_{i=1}^{n-1}\frac{b_i}{n}C_{n-i};\]
and, since $C_n=(-1)^n b_n$, we get the relation
\[b_n(1+n(-1)^n)=-\sum_{i=1}^{n-1}(-1)^ib_i b_{n-i}.\]

\smallskip 
Note that, in the previous computation, we have just used the fact that the product is 
associative and therefore apply for every  associative product defined by 
Equation \eqref{equ.bchperk}. 
More precisely, let $a_n$ be any sequence in $\K$, and for any $(\Gamma,f) \in \mathcal{B}(b\leq a)$,
define
\[
a_{(\Gamma,f)} := \prod_{v \in R(\Gamma)} \frac{a_{d(v)}}{t(v)},
\]
and the product 
\begin{equation}
a\ast b=\sum_{(\Gamma,f)  \in \mathcal{B}(b \leq a)} a_{(\Gamma,f)} Z_\Gamma(f).
\end{equation}

\end{remark}
\begin{proposition}
In the notation above, the  product $\ast$ is associative if and only if there exists
an  $h\in\K$ such that $a_n =h^n b_n$, for every $n>0$.
\end{proposition}

\begin{proof}
 One implication is clear, if $a_n =h^n b_n$, then 
\[
a\ast b=\sum_{(\Gamma,f)  \in \mathcal{B}(b \leq a)} a_{(\Gamma,f)} Z_\Gamma(f)= 
\sum_{(\Gamma,f) } \prod_{v \in R(\Gamma)}h^{d(v)} b_{(\Gamma,f)} Z_\Gamma(f)=
 h^{-1}((ha)\bullet (hb));
\]
this implies that the product $\ast$ is associative 
(in the last equality we use that $\sum_{v \in R(\Gamma)}d(v)=n-1$).
As regards the other implication, assume that the  product $\ast$ is associative; then,
Equation   \eqref{equ.bchperkbis}  holds for the   product $\ast$ instead of $\bullet$. 
Arguing as in the above remark,
we  conclude that the numbers $a_n$ must satisfy Equation~\eqref{equazione bernoulli number},
 and this easily implies that 
$a_n= (-2a_1)^n b_n$, for every $n>0$.

\end{proof}


\begin{thebibliography}{99}





\bibitem{Baker} H. Baker: \emph{Alternants and continuous groups.} Proc. London 
Math. Soc., (2) \textbf{3}, (1905), 24-47.

\bibitem{Campbell}  J. Campbell: \emph{On a law of combination of operators.}
 Proc. London Math. Soc., (1)  \textbf{28}, (1897), 381-390; \textbf{29}, (1898), 14-32.


\bibitem{DK} J.J.~Duistermaat and J.A.C.~Kolk: \emph{Lie Groups.} Springer Universitext (2000).

\bibitem{Dynkin} E. Dynkin:
\emph{Calculation of the coefficients of the Campbell-Hausdorff formula.} Dokl. Akad.
 Nauk., {\bf 57}, (1947), 323-326. An English translation may be found
in: E.B. Dynkin, A.A. Yushkevich, G.M. Seitz, A.L. Onishchik (Eds.), Selected Papers 
of E.B. Dynkin with Commentary, American Mathematical Society/International Press,
Providence, R.I./Cambridge, Mass, (2000).


\bibitem{FioMan}
D. Fiorenza and M. Manetti: \emph{$L_{\infty}$-structures on mapping
cones.} Algebra \& Number Theory, \textbf{1}, (2007), 301-330.


\bibitem{FiMaMa}  D. Fiorenza, M. Manetti and E. Martinengo:
\emph{Semicosimplicial DGLAs in deformation theory.} Preprint
{\texttt{arXiv:0803.0399v1}}.


\bibitem{hall} B.C.~Hall:
\emph{Lie Groups, Lie Algebras, and representations. An elementary introduction.} 
Graduate Texts in Mathematics, {\bf222}, Springer-Verlag, New York Berlin, (2003).

\bibitem{Hausdorff}  F. Hausdorff:
\emph{Die symbolische Exponentialformel in der Gruppentheorie.} Ber. Verh. Saechs.
 Akad. Wiss., Leipzig, Math. Phys. Kl., \textbf{58}, (1906), 19-48.


\bibitem{kathotia} V. Kathotia: \emph{Kontsevich universal formula 
for quantization and the Campbell-Baker-Hausdorff formula.} Internat. 
J. Math. \textbf{11}, (2000), 523-551; \texttt{arXiv:math/9811174v2}.


\bibitem{Konts} M. Kontsevich:
\emph{Deformation quantization of Poisson
manifolds, I.} Letters in Mathematical Physics, \textbf{66}, (2003)
157-216; \texttt{arXiv:q-alg/9709040}.


\bibitem{LodayVallette} 
J.L. Loday and B. Vallette: \emph{Algebraic Operads.} 
Draft version 0.99 (2010); available at the authors web pages.



\bibitem{operads} M.  Markl, S. Shnider and J. Stasheff:
\emph{Operads in algebra, topology and physics}.
Mathematical Surveys and Monographs, {\bf 96}. American Mathematical Society, Providence, RI, (2002).

\bibitem{ore} O.~Ore: \emph{Theory of graphs.} Colloquium publications  
{\bf 38},  American Mathematical Society, Providence, RI, (1962).


\bibitem{informatica} D. Ranjan, E. Pontelli and G. Gupta: \emph{Data structures 
for order-sensitive predicates in parallel nondeterministic system.}
 Acta Informatica, \textbf{37}, (2000), 21-43.


\bibitem{torossian} C. Torossian: \emph{Sur la formule combinatoire de Kashiwara-Vergne.}
J. Lie Theory, \textbf{12}, (2002), 597-616. 


\bibitem{WS} M. Weyrauch  and D. Scholz: \emph{Computing the Baker-Campbell-Hausdorff 
series and the Zassenhaus product.} Computer Physics Communications, \textbf{180}, (2009), 1558-1565.

\end{thebibliography}
\end{document}